\documentclass[11pt]{article}
\usepackage{amsfonts}
\usepackage{mathrsfs}
\usepackage{amsmath}
\usepackage{amssymb}
\usepackage{graphicx}
\graphicspath{{figure/}}
\usepackage{epic}
\usepackage{amsthm,latexsym}
\renewcommand{\paragraph}{\roman{paragraph}}
\newtheorem{theorem}{\scshape \mdseries  Theorem}[section]

\newtheorem{lemma}[theorem]{\scshape \mdseries  Lemma}

\newtheorem{conj}[theorem]{\scshape \mdseries  Conjecture}

\begin{document}

\title{\sf On automorphism group of the reduced finitary power\\ monoid of the additive group of integers   }
\author{ Dein Wong \thanks{Corresponding author, E-mail address:wongdein@163.com. Supported by  NSFC of China (No.12371025). School of Mathematics, China University of Mining and Technology, Xuzhou,  China. },\ \ \ \ \
 \ Songnian Xu \thanks{School of Mathematics, China University of Mining and Technology, Xuzhou,  China. },\ \ \ \ \ Chi Zhang \thanks{School of Mathematics, China University of Mining and Technology, Xuzhou,  China.},\ \ \ \
Zhijun Wang}
%zclqq32@cumt.edu.cn   %707314493@qq.com  zhjxing@imu.edu.cn
\date{}
\maketitle
\noindent {\bf Abstract:}\ \  For a semigroup $S$ with a binary operation $``+"$ (written additively),  denote by $\mathcal{P}_{{\rm fin}} (S)$ the semigroup consisting of all finite nonempty subsets of $S$ and endowed with the setwise binary addition
$$X+Y=\{x+y:x\in X, y\in Y\},$$ which is called the finitary power semigroup of $S$. When $S$ is   monoid with zero $0$,    by $\mathcal{P}_{{\rm fin,} 0}(S)$ we denote  the monoid  consisting of all finite nonempty subsets of $S$ containing $0$, which is  the reduced finitary power monoid of $S$ with the singleton  $\{0\}$ as  zero-element.

 Let $\mathbb{Z}$ be the additive group of all integers and $\mathbb{N}$ the sub-monoid of $\mathbb{Z}$ of all non-negative integers.  For a finite subset $X$ of $\mathbb{Z}$, we denote by   ${\rm max}\ X$ the    maximum member   in $X$.
    %Recently,  Tringali and Yan (\cite{tri2}, J. Combin. Theory Ser. A, 209(2025))  proved that the only non-trivial automorphism of $\mathcal{P}_{{\rm fin,} 0}(\mathbb{N})$
%is the involution $X \mapsto \beta(X) - X$, and they posed a conjecture: {\it The automorphism group of the reduced power monoid $\mathcal{P}_{{\rm fin,} 0}(S)$ of a numerical
%monoid $S$ properly contained in $\mathbb{N}$  must be  the identity}.
 Recently,  Tringali and Yan (\cite{tri2}, J. Comb. Theory, Ser. A, 209(2025))  proved that the only non-trivial automorphism of $\mathcal{P}_{{\rm fin,} 0}(\mathbb{N})$
 is the involution $X \mapsto  {\rm max}\ X - X$.
 Following up on the result  in  \cite{tri2}, Tringali and Wen \cite{triwen} proved that  the automorphism group of the power monoid  $\mathcal{P}_{\rm fin}(\mathbb{Z})$ is isomorphic to $\mathbb{Z}_2  \times {\rm Dih}_{\infty}$, where  ${\rm Dih}_{\infty}$ refers to the  infinite dihedral group. At the end part of \cite{triwen}, Tringali and Wen left a conjecture as follows:
 {\it The only non-trivial automorphism of the reduced finitary power monoid of
$(\mathbb{Z},+)$  is given by $X\mapsto -X$.}
In the present paper, we aim to give a positive proof for the above conjecture.

\section{Introduction}
\quad \quad  Let $H$ be a  semigroup. Denote by $\mathcal{P}_{\rm fin}(H)$ the  (finitary) power semigroup of
$H$, consisting of all  finite non-empty subsets of $H$ and endowed with the setwise binary operation $$   XY := \{xy :\  x \in X, y \in  Y\}.$$
Moreover, if $H$ is a monoid with the  identity $1$,  we denote by $\mathcal{P}_{{\rm fin},1}(H)$ the set of finite subsets of $H$   containing
  $1$; it is a submonoid of $\mathcal{P}_{\rm fin}(H)$ with identity $\{1\}$, henceforth called the reduced (finitary) power monoid of $H$.

 Power semigroups  were first systematically studied by Shafer and Tamura \cite{sha1} in
the late 1960s. A central question of this field is the so called ``Isomorphism Problem
For Power Semigroups": {\it Whether, for semigroups $S$ and $T$ in a certain class $\mathcal{O}$, an isomorphism
between  $\mathcal{P}(S)$ and  $\mathcal{P} (T)$ implies that $S$ and $T$ are isomorphic?}. Although this was answered in the
negative by Mogiljanskaja \cite{mog} for the class of all semigroups,  several other classes have been found for which the answer is positive (see \cite{gan,gar,gou,sha1,tri1}).
Moreover, the isomorphism problem for the
reduced finitary power monoid was answered in the positive for the class of rational Puiseux
monoids and torsion groups \cite{tri3, tri4} and in the negative for commutative valuation monoids
\cite{rag0}. Recently, power monoids   have received more attention. One reason is that  power monoids  provide a natural algebraic framework for a number of important problems in additive combinatorics and related fields  \cite{sar} on the “additive
irreducibility” of the set of quadratic residues modulo a sufficiently large prime. Also, power semigroups    play a central role in the theory of automata and formal languages \cite{alm}.
In recent years, power semigroups  were  investigated from multiple new perspectives   in a series of papers,  such as primality and atomicity \cite{agg};  arithmetic property \cite{ant};    algebraic properties  \cite{bie, dan};  factorization property     \cite{cos, cos1, fan, tri}.

Automorphism groups of   discrete systems, especially  graphs, play an important role when studying the interplay between
group theory and combinatorial theory.
   Early works on automorphism group  of a discrete object   can be traced back to the 1940s, when Hua \cite{hua, hua0} established the theory of matrix geometry, studying   automorphism groups of some graphs defined over matrix spaces, and   Chow \cite{cho} proved that the automorphism group of the Grassmann graph $G(n,k)$ is isomorphic to the projective general linear group $PGL(V)$ if $n\not=2k$.
    % Tang and Wan \cite{tan} determined the automorphism groups of the symplectic graphs over non-degenerate symplectic spaces.

This year, some attention was concentrated on automorphism groups of power monoids or reduced power monoids of certain additive semigroups. Tringali and Yan \cite{tri2}   proved that:
 \begin{theorem}  {\rm (Theorem 3.2, \cite{tri2})} The only   automorphisms of $\mathcal{P}_{\rm fin,0}(\mathbb{N})$ is either the identity or the
 involution $\sigma_0: X \mapsto {\rm max}\ X - X$, where ${\rm max}\ X$ is the maximum member in $X$.\end{theorem}
 The involution  $\sigma_0$ is interesting in the following sense. Every automorphism $f$ of a monoid $H$ can be canonically
extended to an automorphism of $\mathcal{P}_{\rm fin}(H)$  or $\mathcal{P}_{{\rm fin}, 1}(H)$,
$$ X \mapsto f[X]:= \{f(x): x \in X\},\ \ \  \forall  X\in  \mathcal{P}_{\rm fin}(H) \   {\rm or}\ X\in \mathcal{P}_{\rm fin,1}(H),$$ which is referred to  an inner automorphism  of $\mathcal{P}_{\rm fin}(H)$  or $\mathcal{P}_{\rm fin,1}(H)$. Thus the question: {\it Whether $\mathcal{P}_{\rm fin}(H)$  (resp., $\mathcal{P}_{{\rm fin},1}(H)$) has non-inner automorphisms?} is raised naturally. One can see that  the involution $\sigma_0: X \mapsto {\rm max}\ X - X$ for $\forall\ X\in  \mathcal{P}_{{\rm fin},0}(\mathbb{N})$ is not an inner automorphism  of
$\mathcal{P}_{{\rm fin},0}(\mathbb{N})$.

Following up on Theorem 1.1, Tringali and Wen \cite{triwen} proved that  the automorphism group of the finitary power monoid  $\mathcal{P}_{\rm fin}(\mathbb{Z})$ is isomorphic to $\mathbb{Z}_2  \times {\rm Dih}_{\infty}$, where  ${\rm Dih}_{\infty}$ refers to the  infinite dihedral group. At the end part of \cite{triwen}, Tringali and Wen left a conjecture as follows:
 \begin{conj} The only non-trivial automorphism of the reduced finitary power monoid of
$(\mathbb{Z},+)$  is given by $X\mapsto -X$. \end{conj}
In this article, we obtain the following theorem, which   gives a positive proof for the above conjecture.
\begin{theorem} The only non-identity  automorphism of the reduced finitary power monoid of
$(\mathbb{Z},+)$  is given by $X\mapsto -X$ for any $X\in \mathcal{P}_{\rm fin,0}(\mathbb{Z})$.
\end{theorem}

\section{Proof of the main theorem}
\quad\quad Let $\mathbb{Z}$ be the additive group of all integers and $\mathbb{N}$ the sub-monoid of $\mathbb{Z}$ of all non-negative integers. If $i\leq j \in \mathbb{Z}$, we define  $[\![i ,j ]\!] $
to be the discrete interval $\{x \in\mathbb{Z}: i \leq  x \leq  j\}$.  Given $l\in \mathbb{Z}$  and $X\subset\mathbb{Z}$, we set
$$\begin{array}{ccc}l+ X&=& \{l+x: x\in X\};\ \ \  \ \  \ \ \ \ \ \ \ \ \ \ \ \ \ \ \ \ \ \ \ \ \ \ \ \ \ \ \\
 l- X&=& \{l-x: x\in X\}. \ \ \ \ \ \ \ \ \ \ \ \ \ \ \ \ \ \ \ \ \ \ \ \ \ \  \ \ \ \ \ \\
\end{array}$$
 Given  $k\in \mathbb{N}$  we use $kX$ for the $k$-fold sum of $X$, that is  $$kX= \{x_1 + \cdots + x_k : x_1, \ldots, x_k\in X\}.$$

It is easy to see that the automorphism group of $(\mathbb{Z},+) $ contains two elements: one is the
identity automorphism, the other one is given by $x \mapsto -x$ for any $x\in \mathbb{Z}$. Thus  $  \mathcal{P}_{\rm fin,0}(\mathbb{Z})$ has two inner
automorphisms, the only non-identity automorphism, written
as $f_0$, sends any $X\in  \mathcal{P}_{\rm fin,0}(\mathbb{Z})$ to $-X$.
Particularly, $f_0$ permutes $\{-1,0\}$ with $\{0,1\}$. For any     automorphism $f$ of $\mathcal{P}_{{\rm fin},0}(\mathbb{Z})$, we denote $f_0 f$ by $\widehat{f}$ and call it the {\it reversal} of $f$.
Thus, $\widehat{f}$ sends any finite subset $X$ of $\mathbb{Z}$ to $(-X)^f$.

For a finite subset $X\subset \mathbb{Z}$, denote by ${\rm min}\ X$ the minimum member in $X$. Let $f$ be a given automorphism of $\mathcal{P}_{{\rm fin},0}(\mathbb{Z})$ and suppose that $$ {\rm min}(\{0,\ 1\}^f)=-a,\ \ \ \ \   {\rm max} (\{0,\ 1\}^f)=b, \ {\rm where}\ a, b\in \mathbb{N},\  a+b>0;$$
$$ {\rm min} (\{0,-1\}^f)=-c,\ \ \ \ {\rm max} (\{0,-1\}^f)=d,\   {\rm where}\ c, d\in \mathbb{N},\  c+d>0.$$
We study how ${\rm min}(X^f)$ and ${\rm max}(X^f)$ associate to $a, b, c, d$.
\begin{lemma}  For any $X\in     \mathcal{P}_{{\rm fin},0}(\mathbb{Z})$ with ${\rm min }\ X= -x_-$ and ${\rm max}\ X= x_+$, we have $$ \left\{\begin{array}{ccccc} {\rm min} (X^f)&=&-cx_-&-&ax_+,\\
 {\rm max}(X^f)&=& dx_-&+&bx_+.\end{array}\right. $$
\end{lemma}
\begin{proof}
Suppose that  $k\in \mathbb{N}$ is large enough such that $X\subseteq[\![-k,k]\!]$ (whenever $k\geq{\rm max} \{x_-, x_+\}$). Then
$$X+[\![-k,k]\!]=[\![-k-x_-,k+ x_+]\!].$$
Consequently, $$X+k[\![-1,0]\!]+k [\![0,1]\!]= (k+x_-)[\![-1,0]\!]+(k+x_+) [\![0,1]\!].$$
Applying $f$ on two sides of the equality, we have $$X^f+k([\![-1,0]\!]^f)+k ([\![0,1]\!]^f)= (k+x_-)([\![-1,0]\!]^f)+(k+x_+) ([\![0,1]\!]^f).$$
Comparing the maximum members in the sets  involving   the above equality, we have
$$   {\rm max} (X^f)+kd+kb=(k+x_-)d+(k+x_+)b,$$
which follows that $$   {\rm max} (X^f) =d x_-+bx_+.$$
Comparing the minimum members in the sets  involving   the above equality, we have
$$   {\rm min} (X^f)-kc-ka=-(k+x_-)c-(k+x_+)a,$$
which follows that $$   {\rm min} (X^f) =-c x_--ax_+.$$
The proof is completed.
\end{proof}
By considering the  action of $f^{-1}$ on  $\{0,1\}$, we find   that either $f$ or its reversal  fixes $\{0,1\}$  setwise.
\begin{lemma} Let $f$ be an automorphism of $\mathcal{P}_{{\rm fin},0}(\mathbb{Z})$. Then either  $\{0,1\}^f=\{0,1\}$ or $\{0,1\}^{\widehat{f}}=\{0,1\}$.
\end{lemma}
\begin{proof} Since $f$ is a bijective mapping on  $\mathcal{P}_{{\rm fin},0}(\mathbb{Z})$, there exists $X_0\in     \mathcal{P}_{{\rm fin},0}(\mathbb{Z})$ with ${\rm min }\ X_0= -x_-$ and ${\rm max}\ X_0= x_+$,
such that $X_0^f=\{0,1\}$. By Lemma 2.1, we have $$\left\{\begin{array}{ccccc}cx_-&+&ax_+&=&0 \\   dx_-&+&bx_+&=&1, \end{array}\right.$$
where $\{a,b,c,d,x_-,x_+\}\subset \mathbb{N}$ and $\{a+b, c+d, x_-+x_+\}\subset \mathbb{N}\setminus\{0\}$. It follows from $cx_-+ax_+=0$ that $cx_-=ax_+=0$.

The equality $dx_-+bx_+=1$ implies that, either $dx_-=1, bx_+=0$, or $dx_-=0, bx_+=1$.
If the front case happens, we claim that $  d=x_-=1$ and $c=x_+=0$. Indeed, since $dx_-=1$ and both $d$ and $x_-$ are non-negative integers, we have $d=x_-=1$. Moreover, $c=0$ (as $cx_-=0$) and it follows from $(a+b)x_+=0$ (as $ax_+=bx_+=0$) that $x_+=0$ (as $a+b>0$), which proves the claim. Consequently,  $X_0=\{-1,0\}$, confirming that  $\{0,1\}^{\widehat{f}}=\{-1,0\}^f=\{0,1\}$.
If the later case happens, similar discussion gives that  $\{0,1\}^f=\{0,1\}$.
\end{proof}

Before giving a proof for the main theorem, we list some properties for  an automorphism $f$ of   $\mathcal{P}_{{\rm fin},0}(\mathbb{Z})$ with  $\{0,1\}^f=\{0,1\}$.

\begin{lemma} Let $f$ be an   automorphism  of $\mathcal{P}_{{\rm fin},0}(\mathbb{Z})$. If $\{0,1\}^f=\{0,1\}$, then\\
(i)\ \   $\{-1,0\}^f=\{-1,0\}$;\\
(ii) \ \ ${\rm min}({X^f})={\rm min}\  {X} $ and ${\rm max}({X^f})={\rm max}\  {X}$ for any  $X\in \mathcal{P}_{{\rm fin},0}(\mathbb{Z})$;\\
(iii)\ \ $f$  fixes $\{-1,0,2\}$;\\
(iv) \ $f$ fixes each  $[\![-k,l]\!]$ with $k,l\in \mathbb{N}$;\\
(v)\ \ $X^f=X$ for any $X\in \mathcal{P}_{{\rm fin},0}(\mathbb{N})$;\\
(vi)\ $X^f=X$ for any $X\in \mathcal{P}_{{\rm fin},0}(\mathbb{-N})$.
\end{lemma}
\begin{proof} Assume that $Z_0^f=\{-1,0\}$. It follows from   Lemma 2.2 that  $$\left\{\begin{array}{ccccc}cz_-&+&az_+&=&1 \\   dz_-&+&bz_+&=&0 , \end{array}\right.$$ where  ${\rm min }\ Z_0=-z_-$ and ${\rm max}\ Z_0 =z_+$. Since $a=0$ and $b=1$, it is easy to see that   $c=z_-=1$ and $d=z_+=0$, which  proves that   $\{-1,0\}^f=\{-1,0\}$.

It follows from $\{0,1\}^f=\{0,1\}$ and $\{-1,0\}^f=\{-1,0\}$  that $a=d=0$ and $b=c=1$, thus   assertion (ii) follows   immediately from Lemma 2.1.

By assertion (ii), the maximum member in $\{-1,0,2\}^f$ is $2$ and the minimum one   is $-1$.
Thus $f$ sends $\{-1,0,2\}$ either to itself, or to $\{-1,0,1,2\}$. However, the later case is impossible, since  $\{-1,0,1,2\}=\{-1,0\}+\{0,2\}$ and $\{-1,0,2\}$ can not be decomposed into the sum of two non-trivial elements in $\mathcal{P}_{{\rm fin},0}(\mathbb{Z})$.

Note that $[\![-k,l]\!]=k\{-1,0\} +l\{0,1\}$. By applying $f$ to two sides of the equality, we have $$[\![-k,l]\!]^f=k\{-1,0\}^f +l\{0,1\}^f=k\{-1,0\} +l\{0,1\}=[\![-k,l]\!].$$

If $X\in  \mathcal{P}_{{\rm fin},0}(\mathbb{N})$, then ${\rm min}({X^f})={\rm min}\  {X}=0 $ and ${\rm max}({X^f})={\rm max}\  {X}\geq0$,
confirming that $X^f\in  \mathcal{P}_{{\rm fin},0}(\mathbb{N})$.
 By considering the action of $f^{-1}$,  we conclude that  $X^f\in  \mathcal{P}_{{\rm fin},0}(\mathbb{N})$ only if $X\in  \mathcal{P}_{{\rm fin},0}(\mathbb{N})$.
 Consequently,  the restriction of $f$ to  $\mathcal{P}_{{\rm fin},0}(\mathbb{N})$, written as $f^+$, is an automorphism of $\mathcal{P}_{{\rm fin},0}(\mathbb{N})$.
Thus $f^+$ is either the identity automorphism  of   $\mathcal{P}_{{\rm fin},0}(\mathbb{N})$, or $X^f={\rm max}\ {X}-X$ for any $X\in   \mathcal{P}_{{\rm fin,0}}(\mathbb{N})$ (by Theorem 1.1).
If the later case happens, keeping in mind that   $\{0,2,3\}^f=\{0,1,3\}$ and $f$ respectively fixes $\{-1,0,2\}$ and $  [\![-1, 5]\!]$ (by (iii) and (iv)),  we have
$$ \begin{array}{ccc} \{-1,0,1,2,3,5\}&=&\{-1,0,2\}+\{0,1,3\}= \{-1,0,2\}^f+\{0,2,3\}^f \ \ \  \\ &=& (\{-1,0,2\}+\{0,2,3\})^f = [\![-1, 5]\!] ^f=  [\![-1, 5]\!]. \ \ \  \   \end{array}$$
 The contradiction implies that  $f^+$ must be  the identity automorphism  of   $\mathcal{P}_{{\rm fin},0}(\mathbb{N})$  and  it fixes each element $X\in    \mathcal{P}_{{\rm fin},0}(\mathbb{N})$. The proof for (v) is completed.

  Note that $f_0ff_0$ fixes $\{0,1\}$ setwise. Applying (v) to $f_0ff_0$, we have $X^{f_0ff_0}=X$ for any $X\in    \mathcal{P}_{{\rm fin},0}(\mathbb{N})$, which follows that $(-X)^f=-X$ for any $X\in    \mathcal{P}_{{\rm fin},0}(\mathbb{N})$. Consequently,  $Y^f=Y$ for any $Y\in    \mathcal{P}_{{\rm fin},0}(\mathbb{-N})$.
  \end{proof}

Now, we are ready to give a proof for Theorem 1.3. The proof    is essentially an induction on the sum of the boxing dimension of a set $A \in   \mathcal{P}_{{\rm fin},0}(\mathbb{Z})$ and that of  $A^f$ (the technique is mainly borrowed from \cite{tri3}) Given a set $A \subset \mathbb{Z}$, we denote by ${\rm b.dim}(A)$ the smallest integer $k\geq 0$ for which there
exist $k$ (discrete) intervals whose union is $A$, with the understanding that if no such $k$
exists then  ${\rm b.dim}(A) =\infty$. We call   ${\rm b.dim}(A)$ the {\it boxing dimension} of $A$ (see \cite{tri2} for detail).
For example, the boxing dimension of $A=\{-5,-4,-2,0,1,5,6,7\}$ is $4$.

\vskip 3mm
{\bf Proof of Theorem 1.3}
\begin{proof} Let $f$ be an automorphism of $\mathcal{P}_{{\rm fin},0}(\mathbb{Z})$. By Lemma 2.2, either $\{0,1\}^f=\{0,1\}$  or $\{0,1\}^{\widehat{f}}=\{0,1\}$.
If the later case happens, we replace $f$ with  $\widehat{f}$, thus we can always assume that $\{0,1\}^f=\{0,1\}$. In what follows, we only need to prove that $f$ is the identity automorphism.
To achieve the goal, we need to prove that $A^f=A$ for an arbitrary $A\in   \mathcal{P}_{{\rm fin},0}(\mathbb{Z})$.  If $A\subset \mathbb{N}$ or $A\subset -\mathbb{N}$, it has been shown   that $A^f=A$. Now suppose ${\rm min}\  A=a<0$ and ${\rm max}\ A=b>0$.
We  proceed by induction on ${\rm b.dim}(A)+{\rm b.dim}(A^f)$ to complete the proof. If  ${\rm b.dim}(A)+{\rm b.dim}(A^f)=2$, then $A$ has boxing dimension $1$, i.e., $A=[\![a,b]\!]$ is an interval, it has been proved in Lemma 2.3 that $A^f=A$.
Assume  that $X^f=X$   for any  $X \in   \mathcal{P}_{{\rm fin},0}(\mathbb{Z})$ with  ${\rm b.dim}(X)+{\rm b.dim}(X^f)=q\geq  2$ and assume that $A\in   \mathcal{P}_{{\rm fin},0}(\mathbb{Z})$ satisfies    ${\rm b.dim}(A)+{\rm b.dim}(A^f)=q+1$. Let $ {\rm b.dim}(A)=r+1$ and ${\rm b.dim}(A^f)=s+1$. Thus $q+1=(r+1)+(s+1)$.
There exist increasing sequences $a=a_0, a_1, \ldots , a_{2r}, a_{2r+1}=b$  such that
$$A=[\![a_0,a_1]\!]\cup [\![a_2,a_3]\!]\cup \ldots\cup [\![a_{2u},a_{2u+1}]\!]\cup \ldots\cup  [\![a_{2r},a_{2r+1}]\!],$$
where $a_{2u-1}+2\leq a_{2u} \leq a_{2u+1}$  for $u=1,2, \ldots,  r$.
Write $A^f$ as $$A^f=[\![b_0,b_1]\!]\cup [\![b_2,b_3]\!]\cup \ldots\cup [\![b_{2v},b_{2v+1}]\!]\cup \ldots\cup  [\![b_{2s},b_{2s+1}]\!],$$
where  $b_{2v-1}+2\leq b_{2v}\leq b_{2v+1}$  for $v=1,2, \ldots,  s$, $b_0=a_0=a$ and $b_{2s+1}=a_{2r+1}=b$.

If we can prove $A^f=A$, then we are done.
Suppose for a contradiction that $A^f\not=A$. Then there is some $1\leq i\leq  {\rm min}\{2r, 2s\}$ such that $a_i\not=b_i$.
Let $v$ be the least index with $a_v\not=b_v$. Keep in mind that   $[\![x_1,y_1]\!]+[\![x_2,y_2]\!]=[\![x_1+x_2,y_1+y_2]\!]$.

{\sf Case I.} \ \ $v=2u$ for some $1\leq u\leq r$.

In this case, $a_i=b_i$ for $  i=0,\ldots, 2u-1$ and $a_{2u}\not=b_{2u}$. Assume that $a_{2u}<b_{2u}$ (if otherwise, we consider $f^{-1}$). Let $d=a_{2u}-a_{2u-1}-1$.
Then \begin{eqnarray}\begin{array}{ccc} A+[\![0,d]\!]&=&\bigcup\limits_{i=0}\limits^r([\![a_{2i},a_{2i+1}]\!]+[\![0,d]\!])=\bigcup\limits_{i=0}\limits^r[\![a_{2i},a_{2i+1}+d]\!]\\
&=&(\bigcup\limits_{i\in I}[\![a_{2i},a_{2i+1}+d]\!])\ \ \bigcup\ \  [\![a_{2u-2},a_{2u+1}+d]\!],\end{array}\end{eqnarray}
where $I=[\![0,r]\!]\setminus\{u-1,u\}$.
Considering $A^f+[\![0,d]\!]$, we have \begin{eqnarray}  A^f+[\![0,d]\!]=\bigcup\limits_{j=0}\limits^s([\![b_{2j},b_{2j+1}]\!]+[\![0,d]\!])=\bigcup\limits_{j=0}\limits^s[\![b_{2j},b_{2j+1}+d]\!]. \end{eqnarray}
Note that  ${\rm b.dim}(A+[\![0,d]\!])\leq   |I|+1\leq (r-1)+1=r$  and ${\rm b.dim}(A+[\![0,d]\!])^f={\rm b.dim}(A^f+[\![0,d]\!]) \leq s+1$, which implies that $${\rm b.dim}(A+[\![0,d]\!])+{\rm b.dim}(A+[\![0,d]\!])^f\leq r+s+1\leq q.$$ By the induction hypothesis we have \begin{eqnarray}  (A+[\![0,d]\!])^f=A+[\![0,d]\!].\end{eqnarray}
Combining (3) wit $(A+[\![0,d]\!])^f=A^f+[\![0,d]\!]^f, $
we have  $$A+[\![0,d]\!] =A^f+[\![0,d]\!].$$
This is a contradiction, since $a_{2u}\in A\subset A+[\![0,d]\!]$ and $a_{2u}\notin A^f+[\![0,d]\!]$ (note that $b_{2u-1}+d=a_{2u-1}+d=a_{2u}-1<a_{2u}<b_{2u}$).
\vskip 2mm
{\sf Case II.} \ \ $v=2u+1$ for some $1\leq u\leq r-1$.

  In this case, $a_i=b_i$ for $ i=0,1,\ldots, 2u$ and we assume $a_{2u+1}>b_{2u+1}$ (if $a_{2u+1}<b_{2u+1}$ we can replace $f$ with $f^{-1}$ to complete the proof).
Let $$c_0=-a+a_{2u+1}+{\rm max} \{a_{2r}-a_1,b_{2s}-b_1\},\ \   C=[\![-a+a_{2u+1}+1,c]\!]\ {\rm with}\ c\geq c_0.$$
Considering $A+C$, we have $$\begin{array}{ccc}A+C&=&\bigcup\limits_{i=0}\limits^r([\![a_{2i},a_{2i+1}]\!]+[\![-a+a_{2u+1}+1,c]\!])\ \ \ \ \ \ \ \ \ \  \ \ \ \ \  \ \ \ \ \ \ \ \ \ \ \ \ \ \ \ \ \ \ \ \ \ \ \ \\
&\subseteq&[\![a_{2u+1}+1,b+c]\!]\ \ \  \ \ \ \ \ \ \ \ \ \ \ \ \ \ \ \ \ \ \ \ \ \ \ \ \ \ \ \ \ \ \  \ \ \ \ \ \ \ \ \ \ \ \ \ \ \ \ \ \ \ \ \ \ \ \ \  \ \ \ \\\
&\subseteq&[\![a_0+(-a+a_{2u+1}+1),a_1+c]\!]\ \bigcup\ [\![a_1+c+1,b+c]\!]\ \ \ \ \ \ \ \  \ \ \ \ \ \ \ \  \\
&\subseteq&[\![a_0+(-a+a_{2u+1}+1),a_1+c]\!]\ \bigcup\ [\![a_{2r}+(-a+a_{2u+1}+1),b+c]\!]\\
&\subseteq&A+C,\ \ \ \ \ \ \ \ \ \ \ \ \ \ \ \ \ \ \ \ \ \ \ \ \ \ \ \ \ \ \ \ \ \ \ \ \ \ \ \ \ \ \ \ \ \ \ \ \ \  \ \ \ \ \ \ \ \ \ \ \ \ \ \ \ \ \ \ \ \ \ \ \ \ \ \ \end{array} $$
which leads to \begin{eqnarray}A+C=[\![a_{2u+1}+1,b+c]\!].\end{eqnarray}

Analogous computation for $A^f+C$ implies that $$\begin{array}{ccc}A^f+C&=&\bigcup\limits_{j=0}\limits^s([\![b_{2j},b_{2j+1}]\!]+[\![-a+a_{2u+1}+1,c]\!])\ \ \ \ \ \ \ \ \ \  \ \ \ \ \  \ \ \ \ \ \ \ \ \ \ \ \ \ \ \ \ \ \ \ \ \ \ \ \\
&\subseteq&[\![a_{2u+1}+1,b+c]\!]\ \ \  \ \ \ \ \ \ \ \ \ \ \ \ \ \ \ \ \ \ \ \ \ \ \ \ \ \ \ \ \ \ \  \ \ \ \ \ \ \ \ \ \ \ \ \ \ \ \ \ \ \ \ \ \ \ \ \  \ \ \ \\\
&\subseteq&[\![b_0+(-a+a_{2u+1}+1),b_1+c]\!]\ \bigcup\ [\![b_1+c+1,b+c]\!]\ \ \ \ \ \ \ \  \ \ \ \ \ \ \ \  \\
&\subseteq&[\![b_0+(-a+a_{2u+1}+1),b_1+c]\!]\ \bigcup\ [\![b_{2s}+(-a+a_{2u+1}+1),b+c]\!]\\
&\subseteq&A^f+C,\ \ \ \ \ \ \ \ \ \ \ \ \ \ \ \ \ \ \ \ \ \ \ \ \ \ \ \ \ \ \ \ \ \ \ \ \ \ \ \ \ \ \ \ \ \ \ \ \ \  \ \ \ \ \ \ \ \ \ \ \ \ \ \ \ \ \ \ \ \ \ \ \ \ \ \ \end{array} $$
which leads to \begin{eqnarray}A^f+C=[\![a_{2u+1}+1,b+c]\!].\end{eqnarray}

Furthermore, we consider $A+(\{0\}\cup C)$ and $A^f+(\{0\}\cup C)$, respectively. Set $D=\{0\}\cup C.$
Then \begin{eqnarray}\begin{array}{ccc}A+D&=&A\ \bigcup\  (A+C)  \ \ \ \ \ \ \ \  \ \ \ \ \  \ \ \ \ \ \ \ \ \ \ \ \ \ \ \ \ \ \ \ \ \ \ \ \  \ \ \    \\
&=& (\bigcup\limits_{j=0}\limits^r([\![a_{2j},a_{2j+1}]\!]) \ \bigcup \ [\![a_{2u+1}+1,b+c]\!]\ \ \  \ \ \ \  \   \ \\
&=&(\bigcup\limits_{j=0}\limits^{u-1}([\![a_{2j},a_{2j+1}]\!])\ \bigcup \   [\![a_{2u},b+c]\!].\ \ \ \ \  \ \ \ \ \ \ \ \ \ \ \
 \end{array} \end{eqnarray}
 \begin{eqnarray}\begin{array}{ccc}A^f+D&=& A^f\ \bigcup\  (A^f+C)  \ \ \ \ \ \ \ \  \ \ \ \ \  \ \ \ \ \ \ \ \ \ \ \ \ \ \ \ \ \ \ \ \ \ \ \     \\
&=&  (\bigcup\limits_{j=0}\limits^{s}([\![b_{2j},b_{2j+1}]\!])   \bigcup \ [\![a_{2u+1}+1,b+c]\!] \ \ \ \ \ \ \ \ \  \\
&=& (\bigcup\limits_{j=0}\limits^{s-1}([\![b_{2j},b_{2j+1}]\!]) \bigcup \ [\![h,b+c]\!],\ \ \  \  \ \  \ \ \ \ \ \ \ \ \ \ \ \ \   \ \\
  \end{array}\end{eqnarray} where $h={\rm min}\{b_{2s}, a_{2u+1}+1\}$.
It is easy to see  that ${\rm b.dim}(A+D)=  u+1\leq r$ and ${\rm b.dim}(A+D)^f={\rm b.dim}(A^f+D)\leq s+1$, which implies that $${\rm b.dim}(A+D)+{\rm b.dim}(A+D)^f\leq r+s+1\leq q.$$
Then by induction hypothesis we have $A+D=(A+D)^f=A^f+D^f=A^f+D$. This is a contradiction, since $b_{2u+1}+1\in A+D$ however $b_{2u+1}+1\notin A^f+D$ (keeping in mind that $ b_{2u+1}+1\notin A^f$, $ b_{2u+1}+1<a_{2u+1}+1$ and $b_{2u+1}+1<b_{2s}$).

The contradiction implies that $A^f=A$ for any $A\in \mathcal{P}_{\rm fin,0}(\mathbb{Z})$. Consequently, $f$ is the identity automorphism of $\mathcal{P}_{\rm fin,0}(\mathbb{Z})$ if $\{0,1\}^f=\{0,1\}$.

If  $\widehat{f}$ fixes $\{0,1\}$, then $\widehat{f}$  is the identity automorphism of $\mathcal{P}_{\rm fin,0}(\mathbb{Z})$. Thus $f=f_0$ and  $X^f=-X$ for any   $X\in \mathcal{P}_{\rm fin,0}(\mathbb{Z})$.
\end{proof}

\end{document}